\documentclass[12pt,reqno]{article}

\usepackage[usenames]{color}
\usepackage{amssymb}
\usepackage{amsmath}
\usepackage{amsthm}
\usepackage{amsfonts}
\usepackage{amscd}
\usepackage{graphicx}

\usepackage[colorlinks=true,
linkcolor=webgreen,
filecolor=webbrown,
citecolor=webgreen]{hyperref}

\definecolor{webgreen}{rgb}{0,.5,0}
\definecolor{webbrown}{rgb}{.6,0,0}

\usepackage{color}
\usepackage{fullpage}
\usepackage{float}

\usepackage{graphics}
\usepackage{latexsym}
\usepackage{mathtools}
\mathtoolsset{showonlyrefs}

\newcommand{\braces}{\genfrac{\lbrace}{\rbrace}{0pt}{}}

\begin{document}

\theoremstyle{plain}
\newtheorem{theorem}{Theorem}
\newtheorem{corollary}[theorem]{Corollary}
\newtheorem{proposition}{Proposition}
\newtheorem{lemma}{Lemma}
\newtheorem{example}{Examples}
\newtheorem{remark}{Remark}

\begin{center}
\vskip 1cm
{\LARGE\bf
Applications of an identity of Bat\i r
}

\vskip 1cm

{\large
Kunle Adegoke \\
Department of Physics and Engineering Physics, \\ Obafemi Awolowo University, 220005 Ile-Ife, Nigeria \\
\href{mailto:adegoke00@gmail.com}{\tt adegoke00@gmail.com}

\vskip 0.2 in

Robert Frontczak \\
Independent Researcher, 72764 Reutlingen \\ Germany \\
\href{mailto:robert.frontczak@web.de}{\tt robert.frontczak@web.de}
}

\end{center}

\vskip .2 in

\begin{abstract}
Based on an interesting identity of Bat\i r we derive new identities for double sums involving famous number sequences.
We also prove some double sum identities for binomial transform pairs.
\end{abstract}



\bigskip

\section{Introduction}

As a reaction to V\u{a}lean's Master Theorem of Series, Batir~\cite{batir18} established a generalization of this theorem. 
His proof is based on the following identity which holds for any two sequences of complex numbers $(a_k)_{k\geq 1}$ and 
$(b_k)_{k\geq 1}$ \cite[Lemma 1]{batir18}:
\begin{equation}\label{Batir}
\sum_{k=1}^n a_k \sum_{j=1}^k b_j = \sum_{p=0}^{n-1}\sum_{k=1}^{n-p} a_{p+k} b_k.
\end{equation}
 
Our purpose is to derive various double summation identities based on the following identity which is a variation of \eqref{Batir}:
\begin{equation}\label{main}
\sum_{k = 1}^n \sum_{j = 0}^{k - 1} a_{n - j} b_{k - j} = \sum_{k = 1}^n a_k \sum_{j = 1}^k b_j.
\end{equation}
Throughout the paper we use the notation $H_n^{(s)}$, $O_n^{(s)}$, $B_n$, $F_n$ and $L_n$ for generalized harmonic numbers of order $s$ with 
$H_n^{(1)}=H_n$ being the harmonic numbers, generalized odd harmonic numbers of order $s$ with $O_n^{(1)}=O_n$ being the odd harmonic numbers, 
Bernoulli numbers, Fibonacci numbers and Lucas numbers, respectively. We also use $\braces nm$ to denote Stirling numbers of the second kind, $C_n$ the Catalan numbers, and $G_n$ the generalized Fibonacci numbers, the so-called gibonacci numbers.

\section{General results}

\begin{theorem}\label{grat7c2}
If $(a_k)$, $k=1,2,3,\ldots$, is a sequence of complex numbers, then
\begin{equation}\label{sk16ujw}
\sum_{k = 1}^n \sum_{j = 0}^{k - 1} \frac{a_{n - j}}{n - j} = \sum_{k = 1}^n a_k.
\end{equation}
\end{theorem}
\begin{proof}
In~\eqref{main}, set $b_k=1$ and replace $a_k$ with $a_k/k$ for $k=1,2,3,\ldots$
\end{proof}

\begin{theorem}\label{grat7}
If $(a_k)$, $k=1,2,3,\ldots$, is a sequence of complex numbers, then
\begin{equation}
\sum_{k = 1}^n \sum_{j = 0}^{k - 1} (-1)^{k-j+1} a_{n - j} = \sum_{k = 0}^{\lfloor (n-1)/2 \rfloor} a_{2k+1}.
\end{equation}
\end{theorem}
\begin{proof}
In~\eqref{main}, set $b_k=(-1)^k$ and use
$$\sum_{j=1}^k (-1)^j = \frac{1}{2} \left ((-1)^k - 1 \right ).$$
\end{proof}

\begin{proposition}
If $n$ is a non-negative integer, then
\begin{align}
\sum_{k = 1}^n \sum_{j = 0}^{k - 1} \frac{1}{(n - j)^2} = H_n, \\ 
\sum_{k = 1}^n \sum_{j = 0}^{k - 1} \frac{1}{(n - j)\left(2(n - j) - 1 \right)} = O_n.
\end{align}
\end{proposition}
\begin{proof}
Use $a_k=1/k$ and $a_k=(2k-1)^{-1}$ in Theorem~\ref{grat7c2}.
\end{proof}

\begin{proposition}
If $n$ is a non-negative integer, then
\begin{equation}
\sum_{k = 1}^n \sum_{j = 0}^{k - 1} (-1)^{k-j+1} (n-j) = \left ( \Big\lfloor \frac{n-1}{2}\Big\rfloor + 1 \right )^2. 
\end{equation}
\end{proposition}
\begin{proof}
Use $a_k=k$ in Theorem~\ref{grat7} and 
$$\sum_{k=0}^n (2k+1) = (n+1)^2.$$
\end{proof}

\begin{proposition}
If $n$ is a non-negative integer, then
\begin{equation}\label{oddhar_eq1}
\sum_{k = 1}^n \sum_{j = 0}^{k - 1} (-1)^{k-j+1}\frac{1}{n-j} = O_{\lfloor (n-1)/2 \rfloor}. 
\end{equation}
\end{proposition}
\begin{proof}
Use $a_k=1/k$ in Theorem~\ref{grat7}.
\end{proof}

\begin{proposition}
If $n$ is a non-negative integer, then
\begin{equation}
\sum_{k = 1}^n \sum_{j = 0}^{k - 1} (-1)^{k-j+1} F_{n-j} = F_{2(\lfloor (n-1)/2 \rfloor + 1)} 
\end{equation}
and
\begin{equation}
\sum_{k = 1}^n \sum_{j = 0}^{k - 1} (-1)^{k-j+1} L_{n-j} = L_{2(\lfloor (n-1)/2 \rfloor + 1)} - 2.
\end{equation}
\end{proposition}
\begin{proof}
Use $a_k=F_k$ and $a_k=L_k$ in Theorem~\ref{grat7} and combine with
$$\sum_{k=0}^n F_{2k+1} = F_{2(n+1)}\qquad\text{and}\qquad \sum_{k=0}^n L_{2k+1} = L_{2(n+1)}-2.$$
\end{proof}

\begin{proposition}
Let $f(k)$ be any well defined function. If $n$ is a non-negative integer, then
\begin{equation}
\sum_{k = 1}^n \sum_{j = 0}^{k - 1} (-1)^{k-j} f(n-j) B_{n-j} = \frac{f(1)}{2}. 
\end{equation}
\end{proposition}
In particular,
\begin{equation}
\sum_{k = 1}^n \sum_{j = 0}^{k - 1} (-1)^{k-j} B_{n-j}^2 = -\frac{1}{4}. 
\end{equation}
\begin{proof}
Use $a_k=f(k)B_k$ in Theorem~\ref{grat7} and the fact that $B_{2k+1}=0$ for $k\geq 1$.
\end{proof}
Our next result gives a double sum definition of Bernoulli numbers, expressed in terms of the Stirling numbers of the second kind.
\begin{proposition}
If $n$ is a positive integer, then
\begin{equation}
\sum_{k = 1}^n {\sum_{j = 0}^{k - 1} {( - 1)^{n - j} \frac{{\left( {n - j - 1} \right)!}}{{n - j + 1}}\braces{{ n}}{{n - j}}} }  = B_n .
\end{equation}
\end{proposition}
\begin{proof}
Use the well-known result (see~\cite{guo15} for annotated references):
\begin{equation}
B_n  = \sum_{k = 1}^n {\frac{{( - 1)^k k!}}{{k + 1}}\braces{{ n}}{k}} 
\end{equation}
in Theorem~\ref{grat7c2}.
\end{proof}
\begin{theorem}
If $(a_k)$, $k=1,2,3,\ldots$, is a sequence of complex numbers, $r$ is a complex number that is not a negative integer, then
\begin{equation}\label{n1e84dx}
\sum_{k = 1}^n \sum_{j = 0}^{k - 1} \frac{a_{n - j}}{k - j + r} = \sum_{k = 1}^n a_k H_{k + r} - H_r \sum_{k = 1}^n a_k.
\end{equation}
In particular, we have the following double sum identities
\begin{align}
\sum_{k = 1}^n \sum_{j = 0}^{k - 1} \frac{a_{n - j}}{k - j} &= \sum_{k = 1}^n {a_k H_k} \label{rzv2imh},\\ 
\sum_{k = 1}^n \sum_{j = 0}^{k - 1} \frac{a_{n - j}}{2k - 2j - 1} &= \sum_{k = 1}^n {a_k O_k } \label{ml4vyu4} .
\end{align}
\end{theorem}
\begin{proof}
Use $b_j=1/(j+r)$ in~\eqref{main}. Setting $r=0$ in~\eqref{n1e84dx} gives~\eqref{rzv2imh} while $r=-1/2$ produces~\eqref{ml4vyu4}.
\end{proof}

\begin{proposition}
If $n$ is a positive integer, then
\begin{equation}
\sum_{k = 1}^n {\sum_{j = 0}^{k - 1} {\frac{{( - 1)^{n - j} \left( {n - j} \right)!}}{{\left( {n - j + 1} \right)\left( {k - j} \right)}}\braces{{ n}}{{n - j}}} }  =  - \frac{n}{2}B_{n - 1} .
\end{equation}
\end{proposition}
\begin{proof}
Use
\begin{equation}
a_k  = \frac{{( - 1)^k k!}}{{k + 1}}\braces{{ n}}{k}
\end{equation}
in~\eqref{rzv2imh}, noting that~\cite{kellner14}:
\begin{equation}
\sum_{k = 1}^n {\frac{{( - 1)^k k!}}{{k + 1}}\braces{{ n}}{k}H_k}  =  - \frac{n}{2}B_{n - 1} .
\end{equation}
\end{proof}
\begin{proposition}
If $n$ is a non-negative integer, $x$ is a complex number with $x\neq 1$, and  $r$ is a complex number that is not a negative integer, then 
\begin{equation}\label{har_pol_id1}
\sum_{k = 1}^n \sum_{j = 0}^{k - 1} \frac{x^{n-j}}{k - j + r} 
= \frac{1}{1-x} \left (\sum_{k=1}^n \frac{x^k}{k+r} - x^{n+1}(H_{n+r} - H_r) \right ).
\end{equation}
In particular,
\begin{equation}\label{r=0}
\sum_{k = 1}^n \sum_{j = 0}^{k - 1} \frac{x^{n-j}}{k - j} = \frac{1}{1-x} \left ( \sum_{k=1}^n \frac{x^k}{k} - x^{n+1} H_n \right )
\end{equation}
and
\begin{equation}
\sum_{k = 1}^n \sum_{j = 0}^{k - 1} \frac{x^{n-j}}{2k - 2j - 1} = \frac{1}{1-x} \left ( \sum_{k=1}^n \frac{x^k}{2k-1} - x^{n+1} O_n \right ).
\end{equation}
\end{proposition}
\begin{proof}
Work with $a_k=x^k, x\neq 1,$ in \eqref{n1e84dx}. Then 
\begin{align*}
\sum_{k = 1}^n \sum_{j = 0}^{k - 1} \frac{x^{n - j}}{k - j + r} &= \sum_{k = 1}^n x^k H_{k + r} - H_r \sum_{k=1}^n x^k \\
&= \sum_{k = 1}^n x^k H_{k + r} - H_r \frac{x}{1-x} + H_r \frac{x^{n+1}}{1-x}.
\end{align*}
Applying summation by parts we get
\begin{align*}
\sum_{k = 1}^n \frac{x^k}{k + r} &= \sum_{k=1}^n x^k (H_{k+r} - H_{k-1+r}) \\
&= x^{n+1} H_{n+r} - x H_r - \sum_{k=1}^n H_{k+r}\left (x^{k+1} - x^k \right ) \\
&= x^{n+1} H_{n+r} - x H_r + (1-x) \sum_{k=1}^n x^k H_{k+r}.
\end{align*}
Gathering terms and simplifying proves the main statement. The particular cases are obtained for $r=0$ and $r=-1/2$, respectively.
\end{proof}

\begin{remark}
The case $x=1$ can be treated with l'Hospital's rule. However, this case is not interesting in the sense that it is not a double sum 
as we have for $r\geq 0$ an integer
$$\sum_{k = 1}^n \sum_{j = 0}^{k - 1} \frac{1}{k-j+r} = \sum_{k = 1}^n \sum_{j=1}^{k} \frac{1}{j+r} = \sum_{k = 1}^n (H_{k+r}-H_r).$$
The evaluation of the last sum is elementary.
\end{remark}

\begin{corollary}
If $n$ is a non-negative integer and $r$ is a complex number that is not a negative integer, then 
\begin{equation}
\sum_{k = 1}^n \sum_{j = 0}^{k - 1} \frac{G_{n-j}}{k - j + r} = G_{n+2}(H_{n+r} - H_r) - \sum_{k=1}^n \frac{G_{k+1}}{k+r}.
\end{equation}
In particular,
\begin{equation}
\sum_{k = 1}^n \sum_{j = 0}^{k - 1} \frac{G_{n-j}}{k - j} = G_{n+2}H_{n} - \sum_{k=1}^n \frac{G_{k+1}}{k}
\end{equation}
and
\begin{equation}
\sum_{k = 1}^n \sum_{j = 0}^{k - 1} \frac{G_{n-j}}{2k - 2j - 1} = G_{n+2}O_n - \sum_{k=1}^n \frac{G_{k+1}}{2k-1}.
\end{equation}
\end{corollary}
\begin{proof}
Set $x=\alpha$ and $x=\beta$ with $\alpha=(1+\sqrt{5})/2$ and $\beta=-1/\alpha$ in \eqref{har_pol_id1}, simplify, and combine according to the Binet forms for Fibonacci (Lucas) numbers.
\end{proof}

\begin{corollary}
If $n$ is a non-negative integer, then
\begin{equation}
\sum_{k = 1}^n \sum_{j = 0}^{k - 1} \frac{(-1)^j}{k - j}
=\begin{cases}
 \dfrac{1}{2} H_{n/2},& \text{if $n$ is even;}  \\ 
 O_{(n + 1)/2},& \text{if $n$ is odd.}  \\ 
 \end{cases}
\end{equation}
\end{corollary}
\begin{proof}
When $x=-1$ in \eqref{r=0} , then
\begin{align*}
\sum_{k = 1}^n \sum_{j = 0}^{k - 1} \frac{(-1)^{n-j}}{k - j} &= \frac{1}{2}\left (\sum_{k=1}^n \frac{(-1)^k}{k} + (-1)^{n} H_n \right ) \\
&= \frac{1}{2}\left (- H_n + H_{\lfloor n/2 \rfloor} + (-1)^{n} H_n \right ) \\
&= \begin{cases}
 \dfrac{1}{2} H_{\lfloor n/2 \rfloor},& \text{if $n$ is even;}  \\ 
 -\left ( H_n - \dfrac{1}{2} H_{\lfloor n/2 \rfloor}\right ),& \text{if $n$ is odd.}  \\ 
 \end{cases}
\end{align*}
\end{proof}

\begin{proposition}
If $n$ is a non-negative integer and $p$ and $x$ are complex numbers, then 
\begin{equation}\label{har_pol_id2}
\sum_{k = 1}^n \sum_{j = 0}^{k - 1} \frac{x^{n-j}}{k - j}\left (\frac{1}{n+p-j}-\frac{x}{n+1+p-j}\right )
= \sum_{k=1}^n \frac{x^k}{k(k+p)} - H_n \frac{x^{n+1}}{n+1+p}.
\end{equation}
\end{proposition}
\begin{proof}
Multiply \eqref{r=0} by $(1-x)x^{p-1}, p\in\mathbb C,$ and integrate. 
\end{proof}

\begin{corollary}
If $n$ is a non-negative integer and $p$ is a complex number that is not a non-positive integer, then
\begin{equation}\label{puucoix}
\sum_{k = 1}^n \sum_{j = 0}^{k - 1} \frac{1}{(k - j)(n+p-j)(n+1+p-j)} = \frac{1}{p}\left (H_n + H_{p} - H_{n+p}\right ) - \frac{H_n}{n+1+p}.
\end{equation}
In particular,
\begin{equation}\label{whpvty2}
\sum_{k = 1}^n \sum_{j = 0}^{k - 1} \frac{1}{(k - j)(n+1-j)(n+2-j)} = \frac{n}{n+1} - \frac{H_n}{n+2}
\end{equation}
and
\begin{equation}
\sum_{k = 1}^n \sum_{j = 0}^{k - 1} \frac{1}{(k - j)(2(n - j) - 1)(2(n - j) + 1)} = O_n - \frac{n + 1}{2n + 1} H_n.
\end{equation}
\end{corollary}
\begin{proof}
Set $x=1$ in \eqref{har_pol_id2} and calculate
\begin{equation}
\sum_{k=1}^n \frac{1}{k(k+p)} = \frac{1}{p}\sum_{k=1}^n \left ( \frac{1}{k} - \frac{1}{k+p} \right ) = \frac{1}{p} \left (H_n + H_{p} - H_{n+p}\right ).
\end{equation}
Evaluate~\eqref{puucoix} at $p=1$ and $p=-1/2$, respectively.
\end{proof}

\begin{proposition}
If $r$ is a non-negative integer, then
\begin{equation}\label{d5kgtna}
\sum_{m = 1}^r {\sum_{n = 1}^{m - 1} {\sum_{k = 1}^{n - 1} {\sum_{j = 0}^{k - 1} {\frac{1}{{\left( {k - j} \right)\left( {n - j} \right)\left( {n - j + 1} \right)}}} } } }  = r  - H_r - \frac{{r + 1}}{2}\left( {H_r^2  - H_r^{(2)} } \right) .
\end{equation}

\end{proposition}

\begin{proof}
The sums
\begin{align}
\sum_{m = 1}^r {H_m }  &= \left( {r + 1} \right)H_r  - r,\label{awtwabh}\\
\sum_{m=1}^r H_m^2 &= (r+1)H_r^2 - (2r+1)H_r + 2r\label{zjnw2dj},
\end{align}
and
\begin{equation}\label{nfqpja6}
\sum_{m = 1}^r {H_m^{(2)} }  = \left( {r + 1} \right)H_r^{(2)}  - H_r
\end{equation}
can be derived easily. Write $n-1$ for $n$ in~\eqref{whpvty2} and sum to obtain
\begin{equation}\label{de1xari}
\sum_{n = 1}^{m - 1} {\sum_{k = 1}^{n - 1} {\sum_{j = 0}^{k - 1} {\frac{1}{{\left( {k - j} \right)\left( {n - j} \right)\left( {n - j + 1} \right)}}} } }  = m - H_m  - \frac{1}{2}\left( {H_m^2  - H_m^{(2)} } \right).
\end{equation}
Summing~\eqref{de1xari} over $m$ using~\eqref{awtwabh},~\eqref{zjnw2dj} and~\eqref{nfqpja6} produces~\eqref{d5kgtna}.

\end{proof}
\begin{proposition}
If $n$ is a non-negative integer and $p$ is a positive integer, then
\begin{align}
\sum_{k = 1}^n \sum_{j = 0}^{k - 1} \frac{1}{(k - j)(n - j + p)} &= \frac{1}{2}\left( H_n^2 - H_n^{(2)} \right) - \frac{1}{2}\left( H_{p - 1}^2  + H_{p - 1}^{(2)} \right) \nonumber\\
&\qquad + H_{p - 1} \left( H_{p + n - 1} - H_n \right) + H_n \left( H_{p + n} - H_n \right) \nonumber\\
&\qquad\qquad - \sum_{k = 1}^{p - 1} \frac{H_{k - 1}}{n + k}.
\end{align}
In particular,
\begin{equation}
\sum_{k = 1}^n \sum_{j = 0}^{k - 1} \frac{1}{(k - j)(n + 1 - j)} = \frac{1}{2}\left( H_n^2 - H_n^{(2)} \right) + \frac{H_n}{n + 1}.
\end{equation}
\end{proposition}
\begin{proof}
Work with $a_k=1/(k+p),p\geq 1,$ in \eqref{rzv2imh}. The particular case follows by setting $p=1$.
\end{proof}

\begin{proposition}
If $n$ is a non-negative integer, then
\begin{equation}
\sum_{k = 1}^n \sum_{j = 0}^{k - 1} \frac{H_{n-j}}{k-j} = (n+1)H_n^2 - (2n+1)H_n + 2n,
\end{equation}
\begin{equation}
\sum_{k = 1}^n \sum_{j = 0}^{k - 1} \frac{H_{n-j}^2}{k-j} = (n+1)H_n^3 - \frac{3}{2}(2n+1)H_n^2 + 3(2n+1)H_n + \frac{1}{2}H_n^{(2)} - 6n,
\end{equation}
and
\begin{equation}
\sum_{k = 1}^n \sum_{j = 0}^{k - 1} \frac{H_{n-j}^{(2)}}{k-j} = (n+1)H_n H_n^{(2)} - \frac{1}{2}(2n+1)H_n^{(2)} + H_n - \frac{1}{2}H_n^2.
\end{equation}
\end{proposition}
\begin{proof}
Work with $a_k=H_k$, $a_k=H_k^2$, and $a_k=H_k^{(2)}$, in turn, in~\eqref{rzv2imh}. The sums
$$\sum_{k=1}^n H_k^3 = (n+1)H_n^3 - \frac{3}{2}(2n+1)H_n^2 + 3(2n+1)H_n + \frac{1}{2}H_n^{(2)} - 6n$$
and
$$\sum_{k=1}^n H_k H_k^{(2)} = (n+1)H_n H_n^{(2)} - \frac{1}{2}(2n+1)H_n^{(2)} + H_n - \frac{1}{2}H_n^2$$
come from the paper \cite{JinSun}.
\end{proof}

\section{Miscellaneous results involving binomial coefficients}

\begin{proposition}
If $n$ is a non-negative integer, then
\begin{equation}
\sum_{k = 1}^n {\sum_{j = \left\lfloor {n/2} \right\rfloor }^{k - 1} {\frac{1}{{n - j}}\binom{{j}}{{n - j}}} }  = F_{n + 1}  - 1.
\end{equation}
\end{proposition}
\begin{proof}
Use $a_k=\binom{n-k}k$ in Theorem~\ref{grat7c2}, noting that~\cite[Equation (6.130), p.303]{graham}:
\begin{equation}
\sum_{k=0}^n{\binom{n-k}k}=F_{n+1}.
\end{equation}
\end{proof}

\begin{proposition}
If $m$ and $n$ are non-negative integers and $p$ is a complex number, then
\begin{equation}
\sum_{k = 1}^n \sum_{j = 0}^{k - 1} \frac{(-1)^j}{k-j} \binom{p}{j}\binom{n-j}{m} = \binom{n-p}{m-p}\left( H_{n - p} - H_{m - p} + H_m \right).
\end{equation}
\end{proposition}
\begin{proof}
Use
\begin{equation}
a_k = (- 1)^k \binom{p}{n - k}\binom{k}{m}
\end{equation}
in~\eqref{rzv2imh}, since (\cite[Corollary 5]{spiess90})
\begin{equation}
\sum_{k = 1}^n (- 1)^k \binom{p}{n - k}\binom{k}{m} H_k = ( - 1)^n \binom{{n - p}}{{m - p}}\left( H_{n - p} - H_{m - p} + H_m \right).
\end{equation}
\end{proof}

\begin{proposition}
If $n$ is a non-negative integer, then
\begin{equation}\label{rc5dex1}
\sum_{k = 1}^n \sum_{j = 0}^{k - 1} \binom{n}{j} \frac{x^{n - j}}{k - j} = (1 + x)^n H_n - \sum_{k = 1}^n \frac{(1 + x)^{n - k}}{k}.
\end{equation}
In particular, 
\begin{equation}
\sum_{k = 1}^n \sum_{j = 0}^{k - 1} \frac{\binom{n}{j}}{k - j} = 2^n \left(H_n - \sum_{k = 1}^n \frac{1}{2^k k} \right) \label{y3jj1nf}.
\end{equation}
\end{proposition}
\begin{proof}
Work with $a_k=\binom{n}{k} x^k$ in \eqref{rzv2imh} while making use of the Knuth-Boyadzhiev identity
\begin{equation}\label{KB_id}
\sum_{k=1}^n \binom{n}{k} H_k x^k = (1 + x)^n H_n - \sum_{k = 1}^n \frac{(1 + x)^{n - k}}{k}.
\end{equation}
\end{proof}

\begin{proposition}\label{b6x4u19}
If $m$ and $n$ are non-negative integers, then
\begin{equation}
\sum_{k = 1}^n {\sum_{j = 0}^{k - 1} {\frac{{\left( {n - j} \right)!}}{{k - j}}\binom{{n}}{j}\braces{{ m}}{{n - j}}} }  = n^m H_n  - \sum_{k = 1}^n {\frac{{\left( {n - k} \right)^m }}{k}} .
\end{equation}
\end{proposition}
\begin{proof}
Write $-\exp(x)$ for $x$ in~\eqref{rc5dex1} and differentiate $m$ times through the resulting equation and evaluate at $x=0$, using
\begin{equation}
\left. {\frac{{d^m }}{{dx^m }}\left( {1 - e^x} \right)^p } \right|_{x = 0}  = ( - 1)^p p!\braces{{ m}}{p}.
\end{equation}
\end{proof}

\begin{proposition}
If $n$ is a non-negative integer, then
\begin{equation}\label{ejo444m}
\sum_{k = 1}^n \sum_{j = 0}^{k - 1} \frac{\binom{n}{j}^2}{k - j} = \binom{{2n}}{n} \left(2H_n - H_{2n} \right)
\end{equation}
and
\begin{equation}\label{w08gqaa}
\sum_{k = 1}^n \sum_{j = 0}^{k - 1} \frac{1}{k - j}\binom{n}{j}\binom{2n}{n + j} = \binom{3n}{n} H_n - \sum_{k = 1}^n \frac{1}{k}\binom{3n - k}{n - k}.
\end{equation}
\end{proposition}
\begin{proof}
To derive Identity~\eqref{ejo444m}, plug $a_k=\binom nk^2$ in~\eqref{rzv2imh} and use the fact that~\cite{choi11}:
\begin{equation}
\sum_{k=0}^n{\binom nk^2H_k}=\binom{{2n}}{n} \left(2H_n - H_{2n} \right).
\end{equation}
Identity~\eqref{w08gqaa} is obtained by setting $a_k=\binom nk\binom{2n}k$ in~\eqref{rzv2imh} and making use of~\cite{sofo14}:
\begin{equation}
\sum_{k=0}^n{\binom nk\binom{2n}kH_k}=\binom{3n}{n} H_n - \sum_{k = 1}^n \frac{1}{k}\binom{3n - k}{n - k}.
\end{equation}
\end{proof}

\begin{proposition}
If $n$ is a non-negative integer, then
\begin{equation}\label{dixon}
\sum_{k = 1}^n {\sum_{j = 0}^{k - 1} {\frac{{( - 1)^j }}{{n - j}}\binom{{n}}{j}^3 } }   
= \begin{cases}
 ( - 1)^{n/2} \binom{{n}}{{n/2}}\binom{{3n/2}}{n}-1,&\text{if $n$ is even;} \\ 
 1,&\text{if $n$ is odd;}\\ 
 \end{cases} 
\end{equation}
\end{proposition}
\begin{proof}
Set $a_k=\binom nk^3$ in~\eqref{sk16ujw} and use Dixon's identity~\cite{dixon03}:
\begin{equation}
\sum_{k = 0}^n {(-1)^k\binom{{n}}{k}^3 }    
= \begin{cases}
 ( - 1)^{n/2} \binom{{n}}{{n/2}}\binom{{3n/2}}{n},&\text{if $n$ is even;} \\ 
 0,&\text{if $n$ is odd.} 
 \end{cases} 
\end{equation}
\end{proof}

\begin{proposition}
If $n$ is a non-negative integer, then
\begin{equation}
\sum_{k = 1}^n \sum_{j = 0}^{k - 1} \frac{2^{2j}}{k - j} \binom{2(n - j)}{n - j} = 2^{2n + 1} + \left(H_n - 2\right)(2n + 1)\binom{2n}{n}
\end{equation}
and
\begin{equation}
\sum_{k = 1}^n \sum_{j = 0}^{k - 1} \frac{2^{2j}}{2k - 2j - 1} \binom{2(n - j)}{n - j} = \left(O_{n + 1} - 1 \right)(2n + 1)\binom{2n}{n}.
\end{equation}
\end{proposition}
\begin{proof}
Use $a_k=2^{-2k}\binom{2k}k$ in~\eqref{rzv2imh} and also in~\eqref{ml4vyu4} and employ the following identities from~\cite{Adegoke}:
\begin{align*}
\sum_{k = 1}^n 2^{- 2k} \binom{2k}{k} H_k = 2 + \left(H_n - 2 \right)(2n + 1) 2^{- 2n} \binom{{2n}}{n},\\
\sum_{k = 1}^n 2^{- 2k} \binom{2k}{k} O_k = \left(O_{n + 1} - 1 \right)(2n + 1) 2^{- 2n} \binom{{2n}}{n}.
\end{align*}
\end{proof}

\begin{proposition}
If $n$ is a non-negative integer, then
\begin{equation}
\sum_{k = 1}^n \sum_{j = 0}^{k - 1} \frac{2^{2j}}{k - j} \binom{2(n - j)}{n - j} O_{n - j} = \sum_{k = 1}^n \sum_{j = 0}^{k - 1} \frac{2^{2j} }{2k - 2j - 1} \binom{2(n - j)}{n - j} H_{n - j}.
\end{equation}
\end{proposition}
\begin{proof}
This result follows from~\eqref{rzv2imh} and~\eqref{ml4vyu4} and the following identity~\cite{batir21}:
\begin{align*}
\sum_{k = 1}^n 2^{- 2k} \binom{2k}{k} O_k H_k &= (2n + 1) 2^{- 2n} \binom{2n}{n} \left(H_n - 2\right)\left(O_n - \frac{2n}{2n + 1} \right) \\
&\qquad + \frac{2n + 1}{2(4n - 1)} \binom{2n}{n} - 2.
\end{align*}
\end{proof}

\begin{proposition}
If $n$ is a non-negative integer and $s$ is a complex number that is not a negative integer, then
\begin{equation}\label{oq208yi}
\sum_{k = 1}^n \sum_{j = 0}^{k - 1} \frac{(- 1)^j s}{k - j} \binom{s}{n - j} = s\binom{s - 1}{n} H_n + \binom{s - 1}{n} - (- 1)^n.
\end{equation}
\end{proposition}
\begin{proof}
Batir derived the identity~\cite[Corollary 3]{batir21}
\begin{equation}
\sum_{k = 1}^n \binom{n+s}{k} (- 1)^k H_k = (-1)^n \binom{n + s - 1}{n} H_n + \frac{(-1)^n}{n+s} \binom{n + s - 1}{n} - \frac{1}{n+s}.
\end{equation}
Using the transformation $s\mapsto s-n$ this identity can be stated equivalently as
\begin{equation}
\sum_{k = 1}^n \binom{s}{k} (- 1)^k s H_k = s (-1)^n \binom{s - 1}{n} H_n + (-1)^n \binom{s - 1}{n} - \frac{1}{s}.
\end{equation}
Now, set $a_k=(-1)^k\binom{s}{k}$ in~\eqref{rzv2imh}.
\end{proof}

\begin{corollary}
If $n$ is a non-negative integer then
\begin{equation}
\sum_{k=1}^n \sum_{j=0}^{k-1} \frac{1}{k-j}\frac{2^{2j}}{2(n-j)-1} \binom{2(n-j)}{n-j} = 2^{2n+1} - \binom{2n}{n}(H_n+2)
\end{equation}
and
\begin{equation}
\sum_{k=1}^n \sum_{j=0}^{k-1} \frac{2^{2j}}{k-j} \binom{2(n-j)}{n-j} = 2^{2n+1} + (2n+1) \binom{2n}{n}(H_n-2).
\end{equation}
\end{corollary}
\begin{proof}
Evaluate~\eqref{oq208yi} at $s=-1/2$ and $s=1/2$, and keep in mind that
\begin{equation}
\binom{1/2}{r} = (-1)^{r+1} \frac{2^{-2r}}{2r-1} \binom{2r}{r}, \quad \binom{-1/2}{r} = (-1)^r 2^{-2r} \binom{2r}{r},
\end{equation}
and
\begin{equation}
\binom{-3/2}{r} = (-1)^r (2r+1) 2^{-2r} \binom{2r}{r}.
\end{equation}
\end{proof}

\begin{proposition}
If $n$ is a non-negative integer, then
\begin{equation}
\sum_{k = 1}^n {\sum_{j = 0}^{k - 1} {\frac{{( - 1)^j }}{{k - j}}\binom{{n}}{j}^{ - 1} } }  = \left( {n + 1} \right)\left( {\frac{{H_{n + 1} }}{{n + 2}} - \frac{{( - 1)^n  + 1}}{{\left( {n + 2} \right)^2 }}} \right).
\end{equation}
\end{proposition}
\begin{proof}
Use $a_k=(-1)^k\binom nk^{-1}$ in~\eqref{rzv2imh} together with the following identity~\cite{batir21}:
\begin{equation}
\sum_{k = 1}^n {( - 1)^k \binom{{n}}{k}^{ - 1} H_k }  = \left( {n + 1} \right)\left( {\frac{{H_{n + 1} }}{{n + 2}} - \frac{{( - 1)^n  + 1}}{{\left( {n + 2} \right)^2 }}} \right).
\end{equation}
\end{proof}

\begin{lemma}
If $n$ is a non-negative integer and $p$ and $q$ are complex numbers that are not negative integers, then
\begin{equation}\label{wcjr4sk}
\sum_{k = 1}^n {\binom{{k + q}}{p}}  = \frac{{n + q + 1}}{{p + 1}}\binom{{n + q}}{p} - \binom{{q}}{{p + 1}} - \binom{q}{p}
\end{equation}
and
\begin{align}\label{q1iouiz}
\sum_{k = 1}^n {\binom{{k + q}}{p}H_{k + q} }  &= \binom{{n + q}}{p}\left( {\frac{1}{{p + 1}} + \frac{{n + q + 1}}{{p + 1}}\left( {H_{n + q}  - \frac{1}{{p + 1}}} \right)} \right)\nonumber\\
&\qquad - \binom{{q}}{{p + 1}}\left( {H_q  - \frac{1}{{p + 1}}} \right) - \binom{{q}}{p}H_q .
\end{align}
\end{lemma}
\begin{proof}
Identity~\eqref{wcjr4sk} is a consequence of Pascal's formula:
\begin{equation}
\binom{{k + q}}{p} = \binom{{k + q + 1}}{{p + 1}} - \binom{{k + q}}{{p + 1}}.
\end{equation}
Identity~\eqref{q1iouiz} is obtained by differentiating~\eqref{wcjr4sk} first with respect to $p$, and then with respect to $q$ and eliminating $H_{q-p-1}$ between the two resulting equations.
\end{proof}
\begin{lemma}\label{xv2hi0j}
If $k$ is a non-negative integer and $r$ is a complex number that is not a negative integer, then
\begin{equation}
\binom{{k - 1/2}}{k+r} = ( - 1)^r r\binom{{2r}}{r}\frac{{C_k }}{{2^{2k + 2r} }}\binom{{k + r}}{{k + 1}}^{ - 1},
\end{equation}
where $C_j$ is the $j^{th}$ Catalan number, defined for every non-negative integer $j$ by
\begin{equation}
C_j=\frac1{j+1}\binom{2j}j.
\end{equation}
\end{lemma}
\begin{proof}
A consequence of the generalized binomial coefficient, expressed in terms of the Gamma function, $\Gamma(z)$, namely,
\begin{equation}
\binom rs= \frac{{\Gamma (r + 1)}}{{\Gamma (s + 1)\Gamma (r - s + 1)}}.
\end{equation}
\end{proof}
\begin{proposition}
If $n$ is a non-negative integer and $p$ and $q$ are complex numbers that are not negative integers, then
\begin{equation}\label{e4d81bh}
\sum_{k = 1}^n {\sum_{j = 0}^{k - 1} {\frac{1}{{n - j}}\binom{{n - j + q}}{p}} }  = \frac{{n + q + 1}}{{p + 1}}\binom{{n + q}}{p} - \binom{{q}}{{p + 1}} - \binom{{q}}{p}
\end{equation}
and
\begin{align}\label{qukev6y}
\sum_{k = 1}^n {\sum_{j = 0}^{k - 1} {\frac{{\binom{{n - j + q}}{p}}}{{k - j + q}}} }  &= \binom{{n + q}}{p}\left( {\frac{1}{{p + 1}} + \frac{{n + q + 1}}{{p + 1}}\left( {H_{n + q}  - H_q  - \frac{1}{{p + 1}}} \right)} \right)\nonumber\\
&\qquad + \binom{{q}}{{p + 1}}\frac{1}{{p + 1}}.
\end{align}
\end{proposition}
\begin{proof}
Identity~\eqref{e4d81bh} is obtained by putting $a_k=\binom{k+q}p$ in Theorem~\ref{grat7c2} and using~\eqref{wcjr4sk}. Set $a_k=\binom{k+q}p$ in~\eqref{n1e84dx} with $r=q$ and make use of~\eqref{wcjr4sk} and~\eqref{q1iouiz} to get~\eqref{qukev6y}.
\end{proof}

Our next result is a double sum identity that involves binomial coefficients, Catalan numbers, and odd harmonic numbers.
\begin{proposition}
If $n$ is a non-negative integer and $r\ne1/2$ is a complex number that is not a non-positive integer, then
\begin{align}
&\sum_{k = 1}^n {\sum_{j = 0}^{k - 1} {2^{2j} \binom{{n - j + r}}{{r - 1}}^{ - 1} \frac{{C_{n - j} }}{{2k - 2j - 1}}} }\nonumber\\
&\qquad  =  - \binom{{n + r}}{{r - 1}}^{ - 1} \frac{{C_n }}{{2r - 1}}\left( {1 + \left( {2n + 1} \right)\left( {O_n  + \frac{1}{{2r - 1}}} \right)} \right) + \frac{{2^{2n + 1} }}{{\left( {2r - 1} \right)^2 }}.
\end{align}
In particular,
\begin{equation}
\sum_{k = 1}^n {\sum_{j = 0}^{k - 1} {\frac{{2^{2j} C_{n - j} }}{{2k - 2j - 1}}} }  =  - C_n \left( {1 + \left( {2n + 1} \right)\left( {O_n  + 1} \right)} \right) + 2^{2n + 1} .
\end{equation}
\end{proposition}
\begin{proof}
With Lemma~\ref{xv2hi0j} in mind, set $q=-1/2$ and $p=-r-1/2$ in~\eqref{qukev6y}.
\end{proof}
\begin{lemma}[{\cite[Section 13]{Adegoke0}}]
If $n$ is a non-negative integer and $r$ is a complex number that is not a negative integer, then
\begin{equation}\label{ww5fa8g}
\sum_{k = 0}^n {( - 1)^k \binom{{n}}{k}\binom{{k}}{{k + r}}}  
= \begin{cases}
 \dfrac{{\sin (\pi r)}}{{\pi \left( {n + r} \right)}},&\text{if $r\ne 0$;} \\ 
 0,&\text{if $r= 0$, $n\ne 0$;}  
 \end{cases}
\end{equation}
and
\begin{equation}\label{mdaqu57}
\sum_{k = 0}^n {( - 1)^k \binom{{n}}{k}\binom{{k}}{{k + r}}H_{k + r} }  
= \begin{cases}
 \dfrac{{\sin (\pi r)}}{{\pi \left( {n + r} \right)}}\left( {H_{r - 1}  + \dfrac{1}{{n + r}}} \right),&\text{if $r\ne 0$;} \\ 
 -\dfrac1n,&\text{if $r=0$, $n\ne 0$.}  
 \end{cases}
\end{equation}
\end{lemma}

\begin{proposition}
If $n$ is a non-negative integer and $r$ is a complex number that is not a negative integer, then
\begin{align}\label{sin_id}
&\sum_{k = 1}^n \sum_{j = 0}^{k - 1} \frac{(- 1)^j}{k - j + r} \binom{n + r}{j} \nonumber \\  
&\qquad= (-1)^n \binom{n}{n+r}^{-1}
\begin{cases}
 \dfrac{{\sin (\pi r)}}{{\pi \left( {n + r} \right)}}\left( {\dfrac{1}{n + r}-\dfrac1r} \right),&\text{if $r\ne 0$;} \\ 
 -\dfrac1n,&\text{if $r=0$, $n\ne 0$.}  
 \end{cases}
\end{align}
\end{proposition}
\begin{proof}
Set $a_k=( - 1)^k \binom{{n}}{k}\binom{{k}}{{k + r}}$ in~\eqref{n1e84dx} and use~\eqref{ww5fa8g} and~\eqref{mdaqu57}.
\end{proof}

\begin{corollary}
If $n$ is a non-negative integer, then
\begin{equation}
\sum_{k = 1}^n \sum_{j = 0}^{k - 1} \frac{(- 1)^j}{2(k - j) + 1} \frac{\binom{2n+1}{2j} \binom{2j}{j}}{\binom{n}{j}} 2^{-2j}
= (-1)^{n+1} \frac{n(n+1)}{(2n+1)^2} 2^{-2n} \binom{2(n+1)}{n+1}.
\end{equation}
\end{corollary}
\begin{proof}
Set $r=1/2$ in \eqref{sin_id} and simplify using
\begin{equation}
\binom{n+1/2}{j} = \frac{\binom{2n+1}{2j} \binom{2j}{j}}{\binom{n}{j}} 2^{-2j} \quad\text{and}\quad \binom{n}{n+1/2} = \frac{1}{\pi}\frac{2^{2n+2}}{n+1} \binom{2(n+1)}{n+1}^{-1}.
\end{equation}
\end{proof}

\section{Applications to binomial transforms}

Two sequences of complex numbers $(s_n)$ and $(\sigma_n)$ are called a binomial transform pair if these sequences are connected by the relations
\begin{equation}
\sigma_n = \sum_{k = 0}^n \binom{n}{k} (-1)^k s_k \qquad\iff\qquad s_n = \sum_{k = 0}^n \binom{n}{k} (-1)^k \sigma_k.
\end{equation}

\begin{theorem}\label{bintrans}
Let $(a_k)$ be a sequence of complex numbers. Let $(s_k)$ and $(\sigma_k)$ be a binomial transform pair. Then
\begin{equation}
\sum_{k = 1}^n (- 1)^k \sum_{j = 0}^{k - 1} (- 1)^j \binom{k}{j} a_{n - j} s_{k - j} = \sum_{k = 1}^n a_k \sigma_k - s_0 \sum_{k = 1}^n a_k.
\end{equation}
\end{theorem}
\begin{proof}
Work with \eqref{main} and apply it to $b_j=\binom{k}{j}(-1)^j s_j$. Note that
$$\sum_{j=1}^k b_j = \sum_{j=1}^k \binom{k}{j} (-1)^j s_j = \sum_{j=0}^k \binom{k}{j} (-1)^j s_j - s_0 = \sigma_k - s_0.$$
\end{proof}

\begin{proposition}
Let $(a_k)$ be a sequence of complex numbers. Let $F_n$ be the Fibonacci sequence and $L_n$ be the Lucas sequence, respectively. Then
\begin{equation}
\sum_{k = 1}^n \sum_{j = 0}^{k - 1} \binom{k}{j} a_{n - j} F_{k - j} = \sum_{k = 1}^n a_k F_{2k} 
\end{equation}
and
\begin{equation}
\sum_{k = 1}^n \sum_{j = 0}^{k - 1} \binom{k}{j} a_{n - j} L_{k - j} = \sum_{k = 1}^n a_k (L_{2k} - 2).
\end{equation}
\end{proposition}
\begin{proof}
Apply Theorem \ref{bintrans} in turn to the binomial transform pairs $\{s_n,\sigma_n\}=\{(-1)^n F_n, F_{2n}\}$ 
and $\{s_n,\sigma_n\}=\{(-1)^n L_n, L_{2n}\}$.
\end{proof}

\begin{proposition}
Let $(a_k)$ be a sequence of complex numbers. Let $B_n$ be the sequence Bernoulli numbers. Then
\begin{equation}
\sum_{k = 1}^n \sum_{j = 0}^{k - 1} \binom{k}{j} a_{n - j} B_{k - j} = \sum_{k = 1}^n a_k \left ( (-1)^k B_{k} - 1 \right ). 
\end{equation}
\end{proposition}
\begin{proof}
Apply Theorem \ref{bintrans} to the binomial transform pair $\{s_n,\sigma_n\}=\{(-1)^n B_n, (-1)^n B_n\}$. 
\end{proof}

\begin{theorem}
Let $(s_k)$ and $(\sigma_k)$ be a binomial transform pair. Let $m$, $n$, and $r$ be non-negative integers. Then
\begin{align}
&\sum_{k = 1}^n \sum_{j = 0}^{k - 1} \frac{(- 1)^j}{n - j + 1}\binom{n}{j}\sum_{p = 0}^r (- 1)^p \binom{r}{p} s_{k + p + m - j}\nonumber\\
&\qquad= \frac{(- 1)^n}{n + 1} \sum_{p = 0}^m (- 1)^p \binom{m}{p} \left( \sigma_{n + p + r} - \sigma_{p + r} \right).
\end{align}
In particular,
\begin{align}
\sum_{k = 1}^n {\sum_{j = 0}^{k - 1} {\frac{{( - 1)^j }}{{n - j + 1}}\binom{{n}}{j}\sum_{p = 0}^r {( - 1)^p \binom{{r}}{p}s_{k + p - j} } } }  &= \frac{{( - 1)^n }}{{n + 1}}\left( {\sigma _{n + r}  - \sigma _r } \right),\\
\sum_{k = 1}^n {\sum_{j = 0}^{k - 1} {\frac{{( - 1)^j }}{{n - j + 1}}\binom{{n}}{j}s_{k + m - j} } }  &= \frac{{( - 1)^n }}{{n + 1}}\sum_{p = 0}^m {( - 1)^p \binom{{m}}{p}\left( {\sigma _{n + p}  - \sigma _p } \right)} ,
\end{align}
and
\begin{equation}\label{lzb3fn0}
\sum_{k = 1}^n \sum_{j = 0}^{k - 1} \frac{(- 1)^j}{n - j + 1}\binom{n}{j} s_{k - j} = \frac{(- 1)^n}{n + 1}(\sigma_n - \sigma_0).
\end{equation}
\end{theorem}
\begin{proof}
Set $a_k = (- 1)^k \binom{n}{k}(k + 1)^{- 1}$ in~\eqref{main} and apply~\cite[Theorems 6.3 and 7.9]{Adegoke0}.
\end{proof}

\begin{proposition}
If $n$ is a non-negative integer and $x$ is a complex number, then
\begin{equation}\label{fh8f890}
\sum_{k = 1}^n \sum_{j = 0}^{k - 1} \frac{(-1)^j}{n - j + 1} \binom{n}{j} x^{k-j} H_{k - j} = \frac{(- 1)^{n}}{n+1} 
\left ( (1-x)^n H_n - \sum_{k=1}^n \frac{(1-x)^{n-k}}{k} \right ).
\end{equation}
In particular,
\begin{equation}
\sum_{k = 1}^n \sum_{j = 0}^{k - 1} \frac{(-1)^j}{n - j + 1} \binom{n}{j} H_{k - j} = \frac{(- 1)^{n+1}}{n(n+1)}. 
\end{equation}
\end{proposition}
\begin{proof}
Work in~\eqref{lzb3fn0} with the binomial transform pair 
\begin{equation}
s_n = x^n H_n \quad\text{and}\quad \sigma_n = (1-x)^n H_n - \sum_{k=1}^n \frac{(1-x)^{n-k}}{k},
\end{equation}
which follows from the Knuth-Boyadzhiev identity \eqref{KB_id}.
\end{proof}

\begin{proposition}
If $n$ is a positive integer and $r$ is a non-negative integer such that $r<n$, then
\begin{equation}
\sum_{k = 1}^n {\sum_{j = 0}^k {\frac{{( - 1)^j }}{{n - j + 1}}\binom{{n}}{j}\left( {k - j} \right)^r H_{k - j} } }  = \frac{{( - 1)^{n + 1} }}{{n + 1}}\sum_{k = 0}^r {\frac{{( - 1)^k k!}}{{n - k}}\braces{{ r}}{k}} .
\end{equation}
\end{proposition}
\begin{proof}
Write $\exp x$ for $x$ in~\eqref{fh8f890} and proceed as in the proof of Proposition~\ref{b6x4u19}.
\end{proof}
\begin{proposition}
If $n$ is a non-negative integer, then
\begin{align}
&\sum_{k = 1}^n {\sum_{j = 0}^{k - 1} {( - 1)^j 2^{2j - 2k} \frac{{2(k - j) + 1}}{{(n - j + 1)(k - j + 1)}}\binom{{n}}{j}C_{k - j} \left( {H_{k - j + 1}  - O_{k - j + 1} } \right)} }\nonumber\\
&\qquad  = ( - 1)^{n + 1} 2^{ - 2n} \frac{{2n + 1}}{{\left( {n + 1} \right)^2 }}C_n \left( {H_{n + 1} - O_{n + 1} } \right).
\end{align}
\end{proposition}
\begin{proof}
Use the following known result (\cite[Prop. 13.22]{Adegoke0}):
\begin{equation}
\sum_{k = 0}^n {( - 1)^k \binom{{n}}{k}2^{ - 2k} \frac{{2k + 1}}{{k + 1}}C_k \left( {H_{k + 1} - O_{k + 1} } \right)} = - 2^{ - 2n} \frac{{2n + 1}}{{n + 1}}C_n \left( {H_{n + 1} - O_{n + 1} } \right)
\end{equation}
in~\eqref{lzb3fn0}.
\end{proof}

\begin{proposition}
If $n$ is a non-negative integer, then
\begin{equation}
\sum_{k = 1}^n {( - 1)^k \sum_{j = 0}^{k - 1} {\frac{{B_{k - j + 1} }}{{\left( {k - j + 1} \right)\left( {n - j + 1} \right)}}\binom{{n}}{j}} }  = \frac{{( - 1)^n }}{{\left( {n + 1} \right)^2 }}\left( {( - 1)^{n + 1} B_{n + 1}  + \frac{{n - 1}}{2}} \right).
\end{equation}
\end{proposition}
\begin{proof}
Use the identity (\cite[Prop. 13.33]{Adegoke0}):
\begin{equation}
\sum_{k = 0}^n {\binom{{n}}{k}\frac{{B_{k + 1} }}{{k + 1}}}  = ( - 1)^{n + 1} \frac{{B_{n + 1} }}{{n + 1}} - \frac{1}{{n + 1}}
\end{equation}
in~\eqref{lzb3fn0}.
\end{proof}
\begin{proposition}
If $n$ is a non-negative integer, then
\begin{equation}
\sum_{k = 1}^n {\sum_{j = 0}^{k - 1} {\frac{{( - 1)^j }}{{n - j + 1}}\binom{{n}}{j}\binom{{2\left( {k - j} \right)}}{{k - j}}} }  = \frac{1}{{n + 1}}\left( {\sum_{k = 0}^{\left\lfloor {n/2} \right\rfloor } {\binom{{n}}{k}\binom{{n - k}}{k}}  - ( - 1)^n } \right).
\end{equation}
\end{proposition}
\begin{proof}
It is known that~\cite[Prop. 6]{Adegoke1}:
\begin{equation}
\sum_{k = 0}^n {( - 1)^k \binom{{n}}{k}\binom{{2k}}{k}}  = ( - 1)^n \sum_{k = 0}^{\left\lfloor {n/2} \right\rfloor } {\binom{{n}}{k}\binom{{n - k}}{k}} .
\end{equation}
Use this result in~\eqref{lzb3fn0}.
\end{proof}

\begin{proposition}
If $n$ is a non-negative integer, then
\begin{equation}
\sum_{k = 1}^n {\sum_{j = 0}^{k - 1} {\frac{{( - 1)^j }}{{n - j + 1}}\binom{{n}}{j}C_{k - j + 1} } }  = \frac{1}{{n + 1}}\left( {\sum_{k = 0}^{\left\lfloor {n/2} \right\rfloor } {\binom{{n}}{{2k}}C_k }  - ( - 1)^n } \right).
\end{equation}
\end{proposition}
\begin{proof}
Use the result~\cite[Prop. 6]{Adegoke1}:
\begin{equation}
\sum_{k = 0}^n {( - 1)^k \binom{{n}}{k}C_{k + 1} }  = ( - 1)^n \sum_{k = 0}^{\left\lfloor {n/2} \right\rfloor } {\binom{{n}}{{2k}}C_k } 
\end{equation}
in~\eqref{lzb3fn0}.
\end{proof}

\end{document}